\documentclass[10pt]{article}
\usepackage{latexsym}
\usepackage{times}
\usepackage{graphicx}
\usepackage[sf, bf, small]{titlesec}
\usepackage{relsize}
\catcode`@=11
\renewcommand{\@begintheorem}[2]{\it \trivlist
      \item[\hskip \labelsep{\bf #1\ #2{\rm :}}]}
\renewcommand{\@opargbegintheorem}[3]{\it \trivlist
      \item[\hskip \labelsep{\bf #1\ #2\ {\rm (#3)\/:}}]}
\def\@sect#1#2#3#4#5#6[#7]#8{\ifnum #2>\c@secnumdepth
     \def\@svsec{}\else
     \refstepcounter{#1}\edef\@svsec{\csname the#1\endcsname{.}\hskip 1em }\fi
     \@tempskipa #5\relax
      \ifdim \@tempskipa>\z@
        \begingroup #6\relax
          \@hangfrom{\hskip #3\relax\@svsec}{\interlinepenalty \@M #8\par}
        \endgroup
       \csname #1mark\endcsname{#7}\addcontentsline
         {toc}{#1}{\ifnum #2>\c@secnumdepth \else
                      \protect\numberline{\csname the#1\endcsname}\fi
                    #7}\else
        \def\@svsechd{#6\hskip #3\@svsec #8\csname #1mark\endcsname
                      {#7}\addcontentsline
                           {toc}{#1}{\ifnum #2>\c@secnumdepth \else
                             \protect\numberline{\csname the#1\endcsname}\fi
                       #7}}\fi
     \@xsect{#5}}

\@addtoreset{equation}{section}
\newcommand{\Delete}[1]{}

\usepackage{amsmath,amssymb,amsthm}

\theoremstyle{plain}
\newtheorem{Thm}{Theorem}[section]
\newtheorem{Lem}[Thm]{Lemma}

\usepackage{graphpap}
\usepackage{pstricks}
\usepackage{pst-node}

\title{The competition number of a graph \\
with exactly two holes}

\author{
{\sc Jung Yeun LEE }
\hskip-1ex
\thanks{This work was supported by the Korea Research Foundation
Grant funded by the Korean Government (MOEHRD) (KRF-2007-313-C00012).}
\hskip1ex
\quad {\sc Suh-Ryung KIM}
\addtocounter{footnote}{-1}\footnotemark
 \\
[1ex]
Department of Mathematics Education, \\
Seoul National University, Seoul 151-742, Korea. \\
[1ex]
{\sc Seog-Jin KIM} \\
[1ex]
Department of Mathematics Education, \\
Konkuk University, Seoul 143-701, Korea. \\
[1ex]
{\sc Yoshio SANO}
\thanks{The author was supported by JSPS Research Fellowships
for Young Scientists.}
\thanks{Corresponding author. {\em email address}: sano@kurims.kyoto-u.ac.jp}
\\
[1ex]
Research Institute for Mathematical Sciences, \\
Kyoto University, Kyoto 606-8502, Japan.}

\date{}

\begin{document}

\maketitle

\begin{abstract}
Let $D$ be an acyclic digraph. The competition graph of $D$ is a
graph which has the same vertex set as $D$ and has an edge between
$x$ and $y$ if and only if there exists a vertex $v$ in $D$ such
that $(x,v)$ and $(y,v)$ are arcs of $D$. For any graph $G$, $G$
together with sufficiently many isolated vertices is the
competition graph of some acyclic digraph. The competition number
$k(G)$ of $G$ is the smallest number of such isolated vertices.

A hole of a graph is a cycle of length at least $4$ as an induced
subgraph. In 2005, Kim~\cite{compone} conjectured that the
competition number of a graph with $h$ holes is at most $h+1$.
Though Li and Chang~\cite{LC} and Kim {\it et al.}~\cite{lky}
showed that her conjecture is true when the holes do not overlap
much, it still remains open for the case where the holes share
edges in an arbitrary way. In order to share an edge, a graph must
have at least two holes and so it is natural to start with a graph
with exactly two holes. In this paper, the conjecture is proved
true for such a graph.
\end{abstract}


\section{Introduction}

Suppose $D$ is an acyclic digraph
(for all undefined graph-theoretical terms, see \cite{bo} and \cite{book}).
The {\em competition graph} of $D$, denoted by $C(D)$,
has the same vertex set as $D$ and has an edge between
vertices $x$ and $y$
if and only if there exists a vertex $v$ in $D$
such that $(x,v)$ and $(y,v)$ are arcs of $D$.
Roberts~\cite{cn} observed that, for any graph $G$, $G$ together with
sufficiently many isolated vertices is the competition graph of
an acyclic digraph.
Then he defined the {\em competition number} $k(G)$ of a graph $G$
to be the smallest number $k$ such that $G$ together with $k$
isolated vertices added is the competition graph of an acyclic digraph.

The notion of competition graph was introduced by Cohen~\cite{co}
as a means of determining the smallest dimension of ecological
phase space. Since then, various variations have been defined and
studied by many authors (see \cite{kimsu,lu} for surveys). Besides
an application to ecology, the concept of  competition graph can
be applied to a variety of fields, as summarized in \cite{RayRob}.

Roberts~\cite{cn} observed that characterization of competition
graphs is equivalent to computation of competition number. It does
not seem to be easy in general to compute $k(G)$ for a graph $G$,
as Opsut~\cite{op} showed that the computation of the competition
number of a graph is an NP-hard problem (see \cite{kimsu,kr} for
graphs whose competition numbers are known). 
It has been one of the important research problems 
in the study of competition graphs 
to determine the competition numbers 
that are possible for various graph classes. 
A cycle of length at least 4 of a graph as an induced subgraph 
is called a {\em hole} of the graph 
and a graph without holes is called a {\em chordal graph}. 
As Roberts~\cite{cn} showed that the competition number of 
a chordal graph is at most $1$, the competition number of a graph 
with $0$ holes is at most $1$. 
Cho and Kim~\cite{ck} and Kim~\cite{compone} 
studied the competition number of a graph with exactly one hole. 
Cho and Kim~\cite{ck} showed that 
the competition number of a graph with exactly $1$ hole is at most $2$. 

\begin{Thm}[Cho and Kim \cite{ck}]\label{comptwo}
Let $G$ be a graph with exactly one hole.
Then the competition number of $G$ is at most $2$.
\end{Thm}

Kim~\cite{compone} conjectured that the competition number of a graph
with $h$ holes is at most $h+1$ from these results.
Recently, Li and Chang~\cite{LC}
showed that her conjecture is true for a huge family of graphs.
In a graph $G$, a hole $C$ is {\em independent}
if the following two conditions hold for any other hole $C'$ of $G$,
\begin{itemize}
\item[{\rm (1)}]
$C$ and $C'$ have at most two common vertices.
\item[{\rm (2)}]
If $C$ and $C'$ have two common vertices, then they have one common edge
and $C$ is of length at least $5$.
\end{itemize}

\begin{Thm}[Li and Chang \cite{LC}]\label{LCThm}
Suppose that $G$ is a graph with exactly $h$ holes, all of which
are independent. Then $k(G) \leq h+1$.
\end{Thm}
After then, Kim, Lee, and Sano~\cite{lky} generalized 
the above theorem to the following theorem. 

\begin{Thm}[Kim {\it et al.} \cite{lky}]\label{edge-dis}
Let $C_1$, \ldots, $C_h$ be the holes of a graph $G$.
Suppose that
\begin{itemize}
\item[{\rm (1)}]
each pair among $C_1$, \ldots, $C_h$ share at most one edge, and
\item[{\rm (2)}]
if $C_i$ and $C_j$ share an edge,
then both $C_i$ and $C_j$ have length at least $5$.
\end{itemize}
Then $k(G) \leq h+1$.
\end{Thm}

Thus, it is natural to ask if the bound holds when the holes share 
arbitrarily many edges. In this paper, we show that the answer 
is yes for a graph $G$ with exactly two holes. Our main theorem is
as follows.

\begin{Thm}\label{compthree}
Let $G$ be a graph with exactly two holes.
Then the competition number of $G$ is at most $3$.
\end{Thm}

This paper is organized as follows.
In Section 2, we investigate  some properties of graphs with holes. 
In Section 3, we give a proof of Theorem \ref{compthree}.

\section{Preliminaries}

A set $S$ of vertices of a graph $G$ is called a {\em clique} of
$G$ if the subgraph of $G$ induced by $S$ is a complete graph. A
set $S$ of vertices of a graph $G$ is called a {\em vertex cut} of
$G$ if the number of connected components of $G-S$ is greater than
that of $G$.

Cho and Kim~\cite{ck} showed that for a chordal graph $G$, we can
construct an acyclic digraph $D$ with as many vertices of indegree 
$0$ as there are vertices in a clique so that the 
competition graph of $D$ is $G$ with one more isolated vertex: 

\begin{Lem}[\cite{ck}]\label{indegree}
If $X$ is a clique of a chordal graph $G$, then there exists an
acyclic digraph $D$ such that $C(D)=G \cup \{ i \}$ where $i$ is
an isolated vertex, and the vertices of $X$ have only outgoing
arcs in $D$.
\end{Lem}

\begin{Thm}\label{chordalpart}
Let $G$ be a graph and $k$ be a non-negative integer.
Suppose that $G$ has a subgraph $G_1$ with $k(G_1) \leq k$
and a chordal subgraph $G_2$
such that $E(G_1) \cup E(G_2) =E(G)$
and $X:=V(G_1) \cap V(G_2)$ is a clique of $G_2$.
Then $k(G) \le k+1$.
\end{Thm}

\begin{proof}
Since $k(G_1)\le k$, there exists an acyclic digraph $D_1$ such
that $C(D_1)=G_1 \cup I_k$ where $I_k$ is a set of $k$ isolated
vertices with $I_k \cap V(G)= \emptyset$. Since $X$ is a clique of
a chordal graph $G_2$, there exists an acyclic digraph $D_2$ such
that $C(D_2)=G_2 \cup \{ a \}$ where $a$ is an isolated vertex not
in $V(G) \cup I_k$ and that the vertices in $X$ have only outgoing arcs
in $D_2$ by Lemma~\ref{indegree}. Now we define a digraph $D$ as
follows: $V(D)=V(D_1) \cup V(D_2)$ and $A(D)=A(D_1) \cup A(D_2)$.

Suppose that there is an edge in $E(C(D))$ but not in $E(C(D_1))
\cup E(C(D_2))$. Then there exist an arc $(u,x)$ in $D_1$ and an
arc $(v,x)$ in $D_2$ for some $x \in X$. However, this is
impossible since every vertex in $X$ has indegree $0$ in $D_2$.
Thus $E(C(D)) \subseteq E(C(D_1)) \cup E(C(D_2))$.
It is obvious that $E(C(D)) \supseteq E(C(D_1)) \cup E(C(D_2))$
since $E(C(D)) \supseteq E(C(D_i))$ for $i=1, 2$.
Thus 
\[
E(C(D)) = E(C(D_1)) \cup E(C(D_2))=E(G_1) \cup E(G_2)=E(G). 
\]
Hence $C(D)=G \cup I_k \cup \{ a \}$.
Moreover, since $D_1$ and $D_2$ are acyclic, 
$V(G_1) \cap V(G_2) = X$, 
and each vertex in $X$ has only outgoing arcs in $D_2$, it follows 
that $D$ is also acyclic. Hence $k(G) \leq k+1$. 
\end{proof}

\begin{Lem}[\cite{lky}]\label{wheel1}
Let $G$ be a graph and $C$ be a hole of $G$. Suppose that $v$ is a 
vertex not on $C$ that is adjacent to two non-adjacent vertices 
$x$ and $y$ of $C$. Then exactly one of the following is true: 
\begin{itemize}
\item[{\rm (1)}]
$v$ is adjacent to all the vertices of $C$;
\item[{\rm (2)}]
$v$ is on a hole $C^*$ different from $C$ such that there are
at least two common edges of $C$ and $C^*$
and all the common edges are contained in
exactly one of the $(x,y)$-sections of $C$.
\end{itemize}
\end{Lem}


For a graph $G$ and a hole $C$ of $G$, we denote by $X_C$ the set
of vertices which are adjacent to all vertices of $C$.
Note that $V(C) \cap X_C = \emptyset$.
Given a walk
$W$ of a graph $G$, we denote by $W^{-1}$ the walk represented by
the reverse of vertex sequence of $W$.
For a graph $G$ and a hole $C$ of $G$, we call a walk
(resp.\ path) $W$ a {\em $C$-avoiding walk} (resp. {\em
$C$-avoiding path}) if one of the following holds: 
\begin{itemize}
\item{}
$|E(W)| \geq 2$ and
none of the internal vertices of $W$ are in $V(C) \cup X_C$;
\item{}
$|E(W)|=1$ and one of the two vertices of $W$ is not in $V(C) \cup X_C$.
\end{itemize}

The following lemma immediately follows from Lemma~\ref{wheel1}.

\begin{Lem}\label{avoidingvertex}
Let $G$ be a graph and $C$
be a hole of $G$. Suppose that there exists a vertex $v$
such that
$v$ is adjacent to consecutive vertices $v_i$ and $v_{i+1}$ of $C$,
and that
$v$ is not on $X_C$ and not on any hole of $G$.
Then, if there is a $C$-avoiding path $P$ from $v$ to a vertex in
$V(C) \setminus \{v_i,v_{i+1}\}$, then $P$ has length at least $2$.
\end{Lem}

\begin{proof}
Let $P$ be a $C$-avoiding path from $v$ to a vertex $w$ in
$V(C) \setminus \{v_i,v_{i+1} \}$.
If $|E(P)|=1$, then $v$ is adjacent to two non-adjacent vertices of $C$
since $\{v_i, v_{i+1}, w \}$ does not induce a triangle.
Then $v$ satisfies the hypothesis of Lemma~\ref{wheel1}
while it does not satisfy none of (1) and (2) in Lemma~\ref{wheel1},
which is a contradiction. Thus, $|E(P)| \geq 2$.
\end{proof}


\section{Proof of Theorem \ref{compthree}}

In this section, we shall show that the competition number of a
graph with exactly two holes cannot exceed $3$.

Let $G$ be a graph with exactly two holes $C_1$ and $C_2$.
We denote the holes of $G$ by
\[
C_1: v_{0} v_{1} \cdots v_{m-1} v_{0}, \quad C_2: w_{0} w_{1}
\cdots w_{m'-1} w_{0},
\]
where $m$ and $m'$ are the lengths of the holes $C_1$ and $C_2$,
respectively. In the following, we assume that all subscripts of
vertices on a cycle are considered modulo the length of the cycle.
Without loss of generality, we may assume that $m \geq m' \ge 4$.
For $t \in \{1,2\}$, let
\[
X_t := X_{C_t} = \{ x \in V(G) \mid xv \in E(G)
\text{ for all } v \in V(C_t) \}.
\]

In the following, we deal with the case that the two holes have a
common edge since Theorem~\ref{edge-dis} covers the case that the
two holes are edge disjoint.

\begin{Lem}\label{clique1}
If a graph $G$ has exactly two holes $C_1$ and $C_2$,
then both $X_1$ and $X_2$ are cliques.
\end{Lem}

\begin{proof}
Suppose that two distinct vertices $x_1$ and $x_2$ in $X_1$
are not adjacent. Then $x_1 v_0 x_2 v_2 x_1$ and $x_1 v_1 x_2 v_3 x_1$ are
two holes other than $C_1$. That is, $G$ has at least three holes,
which is a contradiction.
\end{proof}

\begin{Lem}\label{path}
Let $G$ be a graph having exactly two holes $C_1$ and $C_2$.
If $C_1$ and $C_2$ have a common edge,
then the subgraph of $G$ induced by $E(C_1) \cap E(C_2)$ is a path.
\end{Lem}

\begin{proof}
Suppose that $G[E(C_1) \cap E(C_2)]$ is not a path.
Without loss of generality,
we may assume that $v_0v_1$ is a common edge but
$v_1v_2$ is not common.
Let $v_i$ be the first vertex on
$C_1$ after $v_1$ common to $C_1$ and $C_2$.
Then $i \in \{2, \ldots, m-2\}$.
Let $w$ be the vertex on $C_2$
that is adjacent to $v_1$
and that is not $v_0$.
Let $Z$ be the $(w,v_i)$-section of $C_2$
which does not contain $v_0$.
Now, consider the $(w,v_{m-1})$-walk
$W:=Z v_{i+1} \cdots v_{m-1}$.
Let $P$ be a shortest $(w,v_{m-1})$-path
among $(w,v_{m-1})$-paths 
such that $V(P) \subseteq V(W)$.
We shall claim that $C:=v_0v_1Pv_0$ is a hole.
Since neither $v_0$ nor $v_1$ is on $W$, none of $v_0$,
$v_1$ is on $P$.
Thus $C$ is a cycle.
By the definition of $P$,
there is no chord between any pair of non-consecutive vertices on $P$.
Since $C_1$ is a hole, $v_0$ is not adjacent to any of
$v_{i+1}$, \ldots, $v_{m-2}$.
Since $\{v_0\} \cup V(Z) \subset V(C_2)$,
$v_0$ is not adjacent to any vertex on $Z$.
Thus $v_0$ is not adjacent to any vertex on $P$.
By a similar argument,
we can show that $v_1$ is not adjacent
to any vertex in $V(P) \setminus \{w\}$.
Hence $C$ is a hole of $G$.
Since $v_1v_2 \not\in E(C)$,
we have $C \neq C_1$
and so $C=C_2$.

If $v_j$ is adjacent to a vertex $v$ on $Z$
for some $j \in \{i+1, \ldots, m-1 \}$,
then $v_j v$ is shorter than any
$(v,v_j)$-path containing $v_i$ in $G[W]$
and so $P$ does not contain $v_i$.
Therefore $v_i \not\in V(C)$,
and so $C \neq C_2$, which is a contradiction.
Thus, $v_j$ is not
adjacent to any vertex on $Z$ for any
$j \in \{i+1, \ldots, m-1\}$.
Hence $v_j$ is not on $Z$ for any $j \in \{i+1, \ldots, m-1\}$.
This implies that no vertex on $W$
repeats and that no two non-consecutive vertices in $W$ are adjacent.
Thus $W=P$.
Then $G[E(C_1) \cap E(C_2)]=v_iv_{i+1} \cdots v_{m-1}v_0v_1$
is a path and we reach a contradiction.
\end{proof}

\begin{Lem}\label{pathx}
Let $G$ be a graph having exactly two holes $C_1$ and $C_2$.
If $|E(C_1) \cap E(C_2)| \geq 2$, then $X_1 = X_2$.
\end{Lem}

\begin{proof}
By Lemma \ref{path},
we have $G[E(C_1) \cap E(C_2)]= w_i w_{i+1} \cdots
w_j$ where $|j-i| \geq 2$.
We take any vertex $x \in X_1$.
If $x \in V(C_2)$, then $C_2$ has a chord $xw_{i+1}$,
which is a contradiction.
Therefore $x \not\in V(C_2)$.
Then $x$ must be contained in $X_2$
by the Lemma~\ref{wheel1} since $x$ is adjacent
to non-adjacent vertices $w_i$ and $w_j$ in $V(C_2)$.
Thus, $X_1 \subseteq X_2$.
Similarly, it can be shown that $X_2 \subseteq X_1$.
\end{proof}

\begin{Lem}\label{noavoid2}
Let $G$ be a graph having exactly two holes $C_1$ and $C_2$.
If there is no $C_t$-avoiding $(u,v)$-path for consecutive
vertices $u$, $v$ on $C_t$ for $t \in \{1,2\}$, then $G-uv$ has
at most one hole.
\end{Lem}

\begin{proof}
First, we consider the case where $uv \not\in E(C_1) \cap E(C_2)$. 
We may assume that $uv$ is an edge of $C_1$. Suppose that $G-uv$
has at least two holes. Let $C^*$ be a hole of $G-uv$ different
from $C_2$. Then $C^*+uv$ contains two cycles $C_1$ and $C'$
sharing exactly one edge $uv$. Note that $C' \neq C_2$ since $uv$
does not belong to $C_2$. If $|E(C')| \geq 4$, then $C'$ is a
hole, which is a contradiction. Thus it follows that $C'-uv$ is a
path of length $2$. Let $x$ be the internal vertex of $C'-uv$.
Since there is no $C_1$-avoiding $(u,v)$-path, it holds that $x
\in X_1$. However, this implies that $C^*$ has a chord joining $x$
and every vertex in $V(C_1) \setminus \{u,v\}$, which is a
contradiction.

Second, we consider the case where $uv \in E(C_1) \cap E(C_2)$.
Then $G-uv$ contains neither $C_1$ nor $C_2$. 
If there exists a vertex $x \in X_1\setminus X_2$ 
(resp.\ $x \in X_2 \setminus X_1$), 
$uxv$ is a $C_2$-avoiding (resp. $C_1$-avoiding) path, 
which is a contradiction. 
Thus we can let $X=X_1=X_2$. 
Suppose that $G-uv$ contains a hole $C^*$. 
Since $C^*$ is not a hole of $G$, 
$uv$ is a chord of $C^*$ in $G$. 
In fact, $uv$ is the unique chord of $C^*$ in $G$. 
Let $Z^*_1$ and $Z^*_2$ be the two $(u,v)$-sections of $C^*$. 
If $|E(Z^*_1)|=|E(Z^*_2)|=2$, then the 
internal vertices $x_1$ and $x_2$ of the $(u,v)$-paths $Z^*_1$ and
$Z^*_2$, respectively, 
are contained in $X$ since there is no hole-avoiding $(u,v)$-path in $G$. 
So $x_1$ and $x_2$ are adjacent
by Lemma~\ref{clique1}, which contradicts the assumption that
$C^*$ is a hole of $G-uv$. If $|E(Z^*_i)|=2$ and $|E(Z^*_j)| \geq
3$ where $\{i,j\}=\{1,2\}$, then the internal vertex $x_i$ of
$Z^*_i$ is in $X$ and $Z^*_j$ is one of $C_1-uv$ and $C_2-uv$
since $Z^*_j+uv$ is a hole of $G$. This implies that the vertex
$x_i$ is adjacent to all the internal vertices of $Z^*_j$, which
also contradicts the assumption that $C^*$ is a hole of $G-uv$.
Hence, $|E(Z^*_1)| \geq 3$ and $|E(Z^*_2)| \geq 3$. This implies
that $C^*$ is composed of $C_1-uv$ and $C_2-uv$ and so $G-uv$ has
at most one hole.
\end{proof}

\begin{Lem}\label{avoid2}
Let $G$ be a graph with exactly two holes $C_1$ and $C_2$ sharing
at least one edge.
Suppose that
there exists a $C_1$-avoiding $(v_i,v_{i+1})$-path
for each $i \in \{ 0, 1, \ldots, m-1 \}$.
Then $G$ has a subgraph $G_1$ which has exactly one hole
and an induced subgraph $G_2$ which is chordal
such that $E(G_1) \cup E(G_2) =E(G)$
and $V(G_1) \cap V(G_2)=X_1 \cup \{ v_j, v_{j+1} \}$
for some $j \in \{ 0, 1, \ldots, m-1 \}$.
\end{Lem}

\begin{proof}
By Lemma~\ref{path}, $G[E(C_1) \cap E(C_2)]$ is a path.
Without loss of generality, we may assume that
$G[E(C_1) \cap E(C_2)]=v_0v_1 \ldots v_k=w_0w_1 \ldots w_k$
for some integer $k \geq 1$.
We let
\[
j=\left\{
\begin{array}{cl}
2 & \text{if } k=1; \\
0 & \text{if } k \ge 2.
\end{array}
\right.
\label{eq8}
\]
Then $\{v_j, v_{j+1}\} \subseteq V(C_1) \setminus V(C_2)$ if $k=1$,
and $\{v_j, v_{j+1}\} \subseteq V(C_1) \cap V(C_2)$ if $k \geq 2$.
Let $L$ be a shortest $C_1$-avoiding $(v_j,v_{j+1})$-path. 
If $|E(L)| \ge 3$, then $L+ v_{j+1} v_j$ is a hole of $G$ sharing exactly
one edge with $C_1$, which is a contradiction. Thus 
$|E(L)|=2$ and so $L=v_jvv_{j+1}$ for some $v\in V(G)\setminus
V(C_1)$. Now we show that $v \not\in V(C_2)$ by contradiction.
Suppose that $v \in V(C_2)$. We first consider the case $k=1$. If
$v=w_{k+1}$, then $v$ is adjacent to two non-adjacent vertices
$v_k (=v_1)$ and $v_{j+1} (=v_3)$ in $V(C_1)$. By
Lemma~\ref{wheel1}, $v$ is in $X_1$ or $G$ has two holes which
have at least two common edges, and we reach a contradiction.
Therefore $v \neq w_{k+1}$. Then $v_{j}$ is adjacent to two
non-adjacent vertices $v_k$ and $v$ in $V(C_2)$, which is also a
contradiction. Thus $v \not\in V(C_2)$ in either case.

Now we will show that $X_1 \cup \{v_j,v_{j+1} \}$ is a vertex cut by
contradiction.
Suppose that $v$ is connected to a vertex in
$V(C_1) \setminus \{v_j,v_{j+1} \}$ by a $C_1$-avoiding path.
Let $v_{\ell}$ be
the first vertex on the $(v_{j+1},v_j)$-path
$C_1-v_jv_{j+1}$ such that there is a
$C_1$-avoiding $(v,v_{\ell})$-path, and
let $P$ be a shortest $C_1$-avoiding $(v,v_{\ell})$-path.
By Lemma~\ref{avoidingvertex},
$|E(P)| \ge 2$. In the following, we will
show that $v_{j+1}$ is adjacent to every internal
vertex on $P$.
Let $Q$ be the
$(v_{j+1},v_{\ell})$-section of $C_1$ which does not contain $v_j$.
Then $v_{j+1}PQ^{-1}$ is a cycle of length at least $4$
different from $C_1$.
Note that
$v_{j+1} \in V(v_{j+1}PQ^{-1})$ while
$v_{j+1} \not\in V(C_2)$ if $k=1$, and
that $v_j \in V(C_2)$ while $v_j \not\in V(v_{j+1}PQ^{-1})$ if $k \ge 2$.
Therefore $v_{j+1}PQ^{-1}$ is also different from $C_2$.
Thus $v_{j+1}PQ^{-1}$ cannot be a hole and so it has a chord.
By the choice of $v_{\ell}$,
no internal vertex of $Q$ is adjacent to any internal vertex of $P$.
Since $P$ is a shortest path,
any two non-consecutive vertices of $P$ are not adjacent.
In addition, since $Q$ is a part of a hole,
any two non-consecutive vertices are not adjacent.
Thus $v_{j+1}$ is adjacent to an internal vertex of $P$.
Let $x$ be the first internal vertex on
$P$ adjacent to $v_{j+1}$ and let $P'$ be the $(v,x)$-section of $P$.
Then $v_{j+1}P'v_{j+1}$ is a hole or a triangle.
However, if $k=1$, then $v_{j+1}P'v_{j+1}$ is different
from $C_1$ and $v_{j+1}$ is not on any hole other than $C_1$.
If $k \ge 2$, then
$v_j \in V(C_1) \cap V(C_2)$
but $v_j$ is not contained in $v_{j+1}P'v_{j+1}$.
Therefore $v_{j+1}P'v_{j+1}$ cannot be a hole
whether $k=1$ or $k \ge 2$.
Thus $v_{j+1}P'v_{j+1}$ is a triangle
and so $x$ immediately follows $v$ on $P$.
Now consider the cycle consisting of $v_{j+1}$,
the $(x,v_{\ell})$-section of $P$,
and $Q^{-1}$. If this cycle is a triangle, then we are done.
Otherwise, we apply the same argument to conclude that $v_{j+1}$
is adjacent to
the vertex immediately following $x$ on $P$.
By repeating this argument,
we can show that $v_{j+1}$ is adjacent to every internal
vertex on $P$. Then the cycle $C'$ consisting of $v_{j+1}$,
the vertex immediately
proceeding $v_{\ell}$ on $P$, $Q^{-1}$ is either a hole or a triangle.
If $k=1$, then $v_{j+1}$ is not on any hole other than $C_1$.
However, $C' \neq C_1$ and so $C'$ cannot be a hole.
If $k \ge 2$, then $v_j$ is not on $C'$
while it is on both $C_1$ and $C_2$, and so $C'$ cannot be a hole.
Thus $C'$ must be triangle and so $\ell=j+2$.

Let $y$ be the last vertex on $P$ that is adjacent to $v_j$. Such
$y$ exists since $v$ is adjacent to $v_j$. Let $P''$ be the
$(y,v_{j+2})$-section of $P$ and $C''$ be the cycle resulting from
deleting $v_{j+1}$ from $C_1$ and then adding path $P''$. Then
$|E(C'')| \ge 4$. If $k=1$, then it holds that $C'' \neq C_1$
since $v_{j+1} \not\in V(C'')$
and that $C'' \neq C_2$ since $v_j \in V(C'')$ and $v_j \not\in V(C_2)$.
If $k \ge 2$, then $C''$ is
different from both $C_1$ and $C_2$ since
$v_{j+1} \not\in V(C'')$.
Thus $C''$ cannot be a hole in either case and so
$C''$ has a chord.
Recall that any two non-consecutive vertices on
$P$ cannot be adjacent and that any two non-consecutive vertices in
$V(C')\cap V(C_1)=V(C_1) \setminus \{v_{j+1}\}$ cannot be
adjacent. Thus a vertex $u$ on $P''$ must be adjacent to a vertex
$v_r$ on $C''$  to form a chord if $k=1$ while a vertex $u$ on
$P''$ must be adjacent to a vertex $v_r \in V(C_1) \setminus
\{v_{j+1}\}$ if $k \ge 2$. Obviously $r \neq j+2$. Moreover, by
the choice of $u$, $r\neq j$. Then $u$ is adjacent to two
nonconsecutive vertices $v_{j+1}$ and $v_r$ on $C_1$. If $k=1$,
then, by Lemma~\ref{wheel1}, $u \in X_1$ or $G$ contains two holes
which have at least two common edges, either of which is a
contradiction. Now suppose that $k \ge 2$. Since $u \not\in X_1$,
by Lemma~\ref{wheel1}, $u$ is on $C_2$ and all the edges common to
$C_1$ and $C_2$ are contained in exactly one of the
$(v_{j+1},v_r)$-section of $C_1$. However, edges $v_jv_{j+1}$ and
$v_{j+1}v_{j+2}$ belong to distinct $(v_{j+1},v_r)$-sections of
$C_1$ even though they are shared by $C_1$ and $C_2$ by the
hypothesis. Thus we have reached a contraction.
Consequently, there is no $C_1$-avoiding path between $v$ and a
vertex in $V(C_1) \setminus \{v_j,v_{j+1}\}$. 
This implies that $X_1 \cup \{v_j,v_{j+1}\}$ is a vertex cut.

Now we define the subgraphs $G_1$ and $G_2$ of the graph $G$ as follows.
Let $Q$ be the component of $G-(X_1 \cup \{v_j,v_{j+1}\})$ that
contains $V(C_1) \setminus \{v_j,v_{j+1}\}$.
Let $G_2$ be the subgraph of $G$
induced by the vertex set $V(G) \setminus V(Q)$.
Then, since $v_0$ (resp. $v_2$) is a vertex in $V(C_1) \cap V(C_2) \cap V(Q)$ 
for $k=1$ (resp. $k \ge 2$), $C_2$ is not contained in
$G_2$ and so $G_2$ is chordal.
Let $G'_1$ be the subgraph induced by $V(Q) \cup X_1 \cup \{v_{j},v_{j+1}\}$.
Then $G_1'$ contains no $C_1$-avoiding $(v_{j},v_{j+1})$-path.
Therefore the subgraph $G_1 := G_1'-v_{j}v_{j+1}$
has exactly one hole by Lemma~\ref{noavoid2}.
By the definitions of $G_1$ and $G_2$,
we can check that $E(G_1) \cup E(G_2) = E(G)$
and $V(G_1) \cap V(G_2) = X_1 \cup \{v_j,v_{j+1}\}$.
Hence the lemma holds.
\end{proof}

Now, we are ready to complete the proof of the main theorem.

\begin{proof}[Proof of Theorem \ref{compthree}]
If $C_1$ and $C_2$ do not share an edge, then $k(G) \le 3$ by
Theorem~\ref{edge-dis}.
Thus we may assume that $C_1$ and $C_2$ share at least one edge.
By Lemma~\ref{path}, $G[E(C_1) \cap E(C_2)]$ is a path.
Suppose that there is no $C_1$-avoiding $(v_i,v_{i+1})$-path
for some $i \in \{0,\ldots,m-1\}$. Then $G_1:=G - v_i v_{i+1} $
has at most one hole by Lemma \ref{noavoid2} and so $k(G_1) \leq
2$ by Theorem~\ref{comptwo}. Let $G_2:=v_i v_{i+1}$. Then $G_2$ is
chordal, $E(G_1) \cup E(G_2) = E(G)$, and $V(G_1) \cap V(G_2)=
\{v_i, v_{i+1} \}$ is a clique of $G_2$. By Theorem
\ref{chordalpart}, we have $k(G) \leq 3$. 

Now we suppose that there is a $C_1$-avoiding $(v_i,v_{i+1})$-path
for any $i \in \{0,1,$ $\ldots, m-1\}$. 
By Lemma~\ref{avoid2},
$G$ has a subgraph $G_1$ which has exactly one hole
and an induced subgraph $G_2$ which is chordal
such that $E(G_1) \cup E(G_2) =E(G)$
and $V(G_1) \cap V(G_2)=X_1 \cup \{ v_j, v_{j+1} \}$
for some $j \in \{ 0, 1, \ldots, m-1 \}$.
Note that $X_1 \cup \{v_j,v_{j+1}\}$ is a clique of $G_2$.
By Theorem~\ref{comptwo}, we have $k(G_1) \le 2$.
Hence $k(G) \le 3$ by Theorem~\ref{chordalpart}.
\end{proof}



\end{document}